\documentclass[reqno,a4paper,12pt]{amsart}
\usepackage{latexsym, amsmath, amssymb, amsthm, a4}
\usepackage{graphicx}
\usepackage{subcaption}

\usepackage{amsfonts, bbold, bbm, marvosym}

\usepackage{mathrsfs}

\usepackage[mathscr]{eucal}

\newcommand{\G}{\Gamma}
\newcommand{\de}{\delta}

\newcommand{\vt}{\vartheta}

\newcommand{\ka}{\kappa}

\newcommand{\m}{\mu}
\newcommand{\n}{\nu}

\newcommand{\ro}{\rho}

\newcommand{\s}{\sigma}
\newcommand{\Si}{\Sigma}
\newcommand{\vs}{\varsigma}
\newcommand{\ta}{\tau}
\newcommand{\f}{\phi}

\newcommand{\vp}{\varpi}

\newcommand{\p}{\psi}

\newcommand{\om}{\omega}

\newcommand{\Om}{\Omega}
\newcommand{\Omb}{\pmb{\Omega}}

\newcommand{\uu}{\upsilon}

\newcommand{\C}{{\mathbb C}}
\newcommand{\R}{{\mathbb R}}

\newcommand{\ab}{{\mathbf a}}

\newcommand{\cb}{\pmb{\mathbf c}}

\newcommand{\hb}{{\mathbf h}}

\newcommand{\mb}{{\mathbf m}}

\newcommand{\rb}{{\mathbf r}}

\newcommand{\Bb}{{\mathbf B}}

\newcommand{\Hb}{{\mathbf H}}
\newcommand{\Ib}{{\mathbf I}}
\newcommand{\Jb}{{\mathbf J}}
\newcommand{\Kb}{{\mathbf K}}

\newcommand{\Xb}{{\mathbf X}}

\newcommand{\aF}{\mathfrak a}

\newcommand{\AF}{\mathfrak A}

\newcommand{\Bc}{{\mathcal B}}
\newcommand{\Cc}{{\mathcal C}}
\newcommand{\Dc}{{\mathcal D}}
\newcommand{\Ec}{{\mathcal E}}

\newcommand{\Gc}{{\mathcal G}}
\newcommand{\Hc}{{\mathcal H}}

\newcommand{\Kc}{{\mathcal K}}
\newcommand{\Lc}{{\mathcal L}}

\newcommand{\Sc}{{\mathcal S}}

\newcommand{\Vc}{{\mathcal V}} 

\newcommand{\Zc}{{\mathcal Z}}

\newcommand{\dist}{{\rm dist}\,}

\newcommand{\grad}{{\rm grad}\,}

\newcommand{\meas}{\operatorname{meas\,}}

\newtheorem{thm}{Theorem}[section]
\newtheorem{cor}[thm]{Corollary}
\newtheorem{lem}[thm]{Lemma}

\theoremstyle{definition}
\newtheorem{defin}[thm]{Definition}

\newtheorem*{rem}{Remark}

\numberwithin{equation}{section}

\usepackage[hyperfootnotes=false,colorlinks=true,allcolors=blue]{hyperref}

\begin{document}


\title[Approximation by solutions of elliptic equations]{Approximation of functions on a compact set  by solutions of elliptic equations. Quantitative results}

\author{Grigori Rozenblum}
\address{Chalmers Univ. of Technology; The Euler International Mathematical Institute and St.Petersburg State Univ.}
\email{$\mathrm{grigori@chalmers.se}$}
\author{Nikolai Shirokov}
\address{St.Petersburg State Univ.;  National Research Univ. Higher School of Economics, St.Petersburg, Russia }
\email{$\mathrm{nikolai.shirokov@gmail.com}$}
\begin{abstract}We establish that a generalized H\"{o}lder continuous  function on an $(m-2)$-Ahlfors regular compact set  in $\R^m$ can be approximated by solutions of an elliptic equation, with the rate of approximation determined by the continuity modulus of the function.
\end{abstract}
\thanks{The work of G.R. was performed at the St. Petersburg Leonhard Euler International Mathematical Institute and supported by
the Ministry of Science and Higher Education (Agreement No. 075–15–2022–287).\\
The work of N.Sh. was supported by  RScF Grant 23-11-00171}
\maketitle
\section{Introduction}\label{sect.intro} Approximation of functions by some better ones is a classical topic in Analysis. We just mention the Weierstrass theorem on real polynomial approximations,  the Lavrentiev-Keldysh-Mergelyan theorem on approximation of continuous  function of a complex variable by polynomials, and Walsh's results on the rational approximation, to name just a few classical results. Up to now, there are numerous studies  of approximation by polynomials, analytic, and harmonic functions, with common feature of these functions being  that they possess a lot of algebraic and analytic structures, which  enrich possible methods of analysis (Math.Sci.Net contains several thousand references on this topic.) A detailed discussion, with a number of references, can be found in the books \cite{Andri}, \cite{Gardiner}, \cite{Dzyadyk}, \cite{Tamrazov}; some most recent developments are presented, in particular,  in   papers  by P.~Paramonov, P.~Gauthier and their co-operators (see, e.g., their latest publications \cite{Paramonov}, \cite{Gauthier} and references therein).

Along with this direction, there is an interest in approximations by means of more special objects,  in particular, by  solutions of rather general differential  equations, which possess considerably fewer structures, therefore the approaches developed for the former type of problems fail. A fundamental result was obtained here by F.Browder  in \cite{Browder}, \cite{Browder1}, where  for continuous functions on a compact set of zero Lebesgue measure, there was established the   property of  approximation   by solutions of a very general class of differential equations, including the elliptic ones or order $\rb$, having  the form $\Lc u\equiv \Lc (X,D_X)u=0.$ The conditions imposed on the operator $\Lc$ were rather permissive, namely, the coefficients of the operator $\Lc$ as well the ones of its formal dual $\Lc',$ should belong to $C^1$ in some domain $\Omb$, with $\Lc'$ possessing the property of the local uniqueness in $\Omb$ for  the Cauchy problem. In this setting, using advanced machinery of Functional Analysis, F.Browder proved that for any compact set $\Kb\subset\Omb,$ such that $\Omb\setminus \Kb$ is connected, any continuous function on $\Kb$ can be approximated in $C(\Kb)$ by $C^{\rb}$ solutions of the equation $\Lc u =0.$
This extraordinary  result was later extended in the direction of further relaxing the conditions imposed on the structure of the operator and its coefficients. However, certain questions, considered in the study of polynomial and analytic approximations, remained untouched upon in \cite{Browder}, \cite{Browder1} and further studies. Namely, these questions concern finding a relation between the quality of the approximated function and the quality  (say, the convergence rate) of the approximation.  For the traditional approximation methods, there are classical results of this kind. One should mention here, of course, the results by   D.~Jackson and S.~Bernstein   connecting the smoothness class  of a function on an interval and the rate of its polynomial approximation. Such results have been later called 'the constructive description of functional classes.' Further developments by S.~Nikolski, V.~Dzyadyk, I.~Shevchuk, V.~Belyi, V.~Andrievski, N.~Lebedev,   and  many others dealt with such constructive description of various classes of functions in terms of approximation by algebraic and trigonometrical polynomials, analytical and entire functions, but not by more general ones.

Relatively recently, a progress was made in finding such constructive description of H\"older classes in terms of approximation by harmonic functions, which required, as it was explained above for the qualitative approximation problem, a number of new ideas.

In the present paper, we combine two problems described above. For  a compact set $\Kb$ of zero Lebesgue measure (more exactly, an Ahlfors-David $(m-2)$-regular set), we find a necessary and sufficient condition for a function to  belong to a prescribed generalized H\"older class in terms of uniform approximation by solutions of a given elliptic second order divergence form equation   with Dini continuous coefficients. We use some ideas developed recently in \cite{AlSh} and \cite{Pavlov} where the approximation by harmonic functions was considered. Here, it turns out that the sufficient  condition in constructive terms for a function to belong to a H\"older  class is proved rather elementarily. On the opposite, the proof of the necessary condition, namely that a function in the class under consideration can, in fact, be approximated by  solutions of the equation, requires some tricky calculations. The advantage of our results here, compared, say, with the ones in \cite{Browder}, is that we not only prove the existence of approximating solutions but  present these approximations explicitly  and give  sharp error estimates.

To explain, in the introductory manner, our results, we recall the general philosophy of approximation of functions. Having a 'bad' function $f$, one approximates it by 'nice' functions $v_\de$ in some class. The closer $v_\de$ is to $f$ in some metric, the worse is the approximating function $v_\de,$ in particular, its norm in a proper  space may grow while $\de\to 0.$ Quantitative approximation results specify  the correspondence  between $\de$ and this worsening of the approximating function. In our case, the function $f$ is continuous on a compact set $\Kb\subset\R^m,$ $m\ge3,$ moreover, it belongs to some generalized H{\"o}lder class $C^\om(\Kb)$ corresponding to a continuity modulus $\om(t).$ The worsening of the approximating function  $v_\de$ is measured by the smallness of the $\de$-neighborhood $\Kb_\de$ of $\Kb,$ where $v_\de$ is 'nice' in the proper sense, and by the size of the gradient of $v_\de$ on $\Kb_\de.$ In this setting, the main results of the paper are as follows (the detailed formulations can be found  further on in the text.)

\textbf{The sufficient condition:} If for a function $f\in C(\Kb)$ for any sufficiently small $\de>0,$ there exists a function $v_\de$ on $\Kb_\de$ such that
 \begin{equation}\label{11}
 |f(X)-v_\de(X)|\le C\om(\de) \, \mbox{on}\, \Kb\,  \mbox{and}\, |\nabla(v_\de)(X)|\le C'\frac{\om(\de)}{\de}\, \mbox{on}\, \Kb_\de.
 \end{equation}
  then $f\in C^{\om}(\Kb).$

\textbf{The necessary condition}. If $\Kb$ is sufficiently regular, for example, is an  a Lipschitz surface of codimension $2$ and  $f\in C^{\om}(\Kb)$ (or, more generally, is Aflfors $(m-2)$-regular), then for any small $\de,$  there exists a function $v_\de$ on $\Kb_\de$ satisfying \eqref{11} and being, additionally, the weak solution in $\Kb_\de$ of a fixed second order elliptic equation $\Lc u=0$ in divergence form, with Dini continuous coefficients.

Being put together, these two results can be understood as an example of a 'self-improving' property. Namely, if a function $f$ can be approximated in the sense of  \eqref{11} by \emph{some } functions $v_\de,$ it can be, in fact, approximated, in the same sense and with the same quality, by  essentially more special functions,  solutions of the prescribed  elliptic equation.

Shortly on the structure of the paper. In the second section we present some known  basic facts we need, concerning the Hau{\ss}dorff measure and Ahlfors-David regular sets. Although the results of the paper are valid for a more general class of compact sets, the Ahlfors-David regular ones admit a more explicit and visual description.  We describe  here some key facts about the Green function for elliptic equations with moderately regular coefficients as well.  Then, in Sect.3,  we establish the sufficient approximation  condition.  The essential part of our further considerations is based on the construction of an extension of a continuous function from a compact set $\Kb$ to its neighborhood, with control over the derivatives of the extension; this construction is described in Sect.4. After deriving in Sect.5 some important estimates   of integrals involving the continuity modulus, in the next sections we justify the integral representation for the extended function using the Green function, and finally present  the construction of the approximating solutions of the elliptic equation and  establish the estimates for the solution, which justifies the necessity part of the approximation theorem.

Our reasoning concerns approximation of an individual H{\"o}lder continuous function. It is interesting to extend our considerations in order to study the metric properties of the approximating operator. While we consider the functions on a compact set which is Ahlfors $(m-2)$-regular (since it directly extends the previously studied case of a curve in $\R^3$), certain modifications of our approach enable one to handle $\theta$-regular sets of any order $\theta<{m-1}.$ We plan to treat these  and other related topics in further publications.

\section{Preliminaries}\label{sect.prel}
\subsection{Hau{\ss}dorf measure} We recall the definition of the Hau{\ss}sdorff measure (see, e.g. \cite{Falconer}). We denote by $\wr U\wr$ the diameter of the set $U\in \R^m.$  Let $\Ec\subset \R^m$ be a bounded set. A finite or countable collection of subsets  $U_k\subset \R^m$ is called $\de$-cover of $\Ec,$ $\de>0,$ if $\wr U_k\wr\le \de$ and $\Ec\subset\cup U_k.$  For a fixed $\vt\ge 0$ and $\de>0$, we denote by $\Hc_\de^\vt(\Ec)$ the quantity
\begin{equation*}
  \Hc^\vt_\de(\Ec)=\inf\left\{\sum_{k}\wr U_k\wr^\vt: \{U_k\} \, \mbox{is}\,\mbox{a}\,\de\mbox{-cover}\,\mbox{of}\, \Ec\right\}.
\end{equation*}
It is not excluded that $\Hc_\de^\vt(\Ec)$ equals zero. As $\de\searrow 0,$ this infimum increases and tends to a limit, which is allowed to be zero or infinity. This limit,
\begin{equation*}
  \Hc^\vt(\Ec)=\lim_{\de\to 0} \Hc^\vt_\de(\Ec),
\end{equation*}
is called  the $\vt$-dimensional Hau{\ss}dorff measure of the set $\Ec,$
It is known, see, e.g., \cite{Federer}, that $\Hc^\vt$ is a regular Borel  measure in $\R^m$.

It follows from the definition that for any set $\Ec,$ there exists a unique $\theta\ge0$ such that  $\Hc^\vt(\Ec)=\infty$ for $\vt<\theta$ and $\Hc^\vt(\Ec)=0$ for $\vt>\theta$. This number $\theta$ is called the Hau{\ss}dorff dimension of $\Ec,$ $\theta:=\dim_{\Hc}(\Ec)$. Only sets of Hau{\ss}dorff dimension $\theta$ (but not all of them by far) may have finite positive Hau{\ss}dorff measure $\Hc^\theta$. Discrete sets (but not only they) have zero Hau{\ss}dorff dimension.

The Hau{\ss}dorff measure is a convenient generalization of the Lebesgue measure. For $\theta=m$, it coincides with the $m$-dimensional Lebesgue measure, up to a constant factor, determined by the dimension. For subsets in a compact Lipschitz surface $\Si$ of dimension $\theta,$ $1\le \theta< m,$ the Hau{\ss}dorff measure, again, up to a constant factor, coincides with the surface measure on $\Si$, generated in the standard way by the Lebesgue measure in $\R^{m}.$  The Hau{\ss}dorff measure is invariant with respect to shifts and rotations; it is also homogeneous under a dilation $\Dc_r$ with ratio $r$:
 \begin{equation*}
   \Dc_r: x\mapsto r x, r>0:
 \end{equation*}
$\Hc^\theta(\Dc_r(\Ec))=r^{\theta}\Hc^\theta(\Ec).$  More generally, under a bi-Lipschitz mapping the Hau{\ss}dorff measure changes in a controllable way: if $ c|x-y|\le|F(x)-F(y)|\le C|x-y|$ then
\begin{equation*}
  c^\theta \Hc^\theta(\Ec)\le \Hc^\theta(F(\Ec))\le C^\theta\Hc^\theta(\Ec).
\end{equation*}
\subsection{Ahlfors regular sets}
The union of a finite collection of sets having  Hau{\ss}dorff dimension $\theta$ is, again, a set of this  Hau{\ss}dorff dimension. On the other hand, the union of finitely many  sets of different Hau{\ss}dorff dimension has Hau{\ss}dorff dimension coinciding with the largest one of these sets. So, generally, a set may have pieces of different Hau{\ss}dorff dimension. The following definition describes
sets which have one and the same Hau{\ss}dorff dimension everywhere.

\begin{defin}\label{DefAlfors}
 Let, for a point $X\in\R^{m},$  $B_r(X)$ denote the open  ball with center $X$  and radius $r.$ The set $\Ec\subset\R^m$ is called 'Ahlfors regular of dimension $\theta$' (called sometimes 'Ahlfors-David regular', or, shorter, '$\theta$-regular') if there exist positive  constants $C_{\pm}=C_{\pm}(\Ec,\theta)$ such that for any $X\in \Ec,$
 \begin{equation*}
   C_-r^{\theta}\le \Hc^{\theta}(\Ec\cap B_r(X) )\le C_+ r^{\theta}, 0<r\le\wr\Ec\wr.
 \end{equation*}
\end{defin}
In particular, this property  is preserved under bi-Lipshitz mappings.
It can be considered as a possible convenient generalization of the notion, for $\theta=1$, of chord-ark curves, see, e.g., \cite{AlSh}. Compact Lipschitz surfaces, as well as relatively open sets on them, are,  obviously, Ahlfors regular. Ahlfors regular sets arise in many topics of Analysis. One can find a discussion of application of this notion, e.g., in \cite{DS1}., \cite{JW}, \cite{MatBook}, etc.

\subsection{A cover lemma} Cover lemmas play traditionally an important role in Analysis.  We will use the  following cover lemma for Ahlfors regular sets, which  can be found  in \cite{Mattila}, Lemma 2.1.

\begin{lem}\label{Lem.cover} Let $\Ec\subset\R^m$ be a compact $\theta$-regular set. Then for some positive constant  $C(\Ec)$,  any $r<R$, $r,R\in (0,\wr\Ec\wr),$ and any point $X_0\in \Ec,$ there exists a finite collection of points $\Zc=\{Z_l, \, l=1,\dots, \mb\},$ such that
\begin{equation}\label{mcontrol}
  (5^\theta C(\Ec))^{-1}(Rr^{-1})^\theta\le \mb\le 2^\theta C(\Ec)(Rr^{-1})^\theta,
\end{equation}
all points $Z_l$ are contained in $\Ec\cap B_R(X_0)$, the balls $B_r(Z_l)$ are disjoint and the balls $B_{5r}(Z_l)$ cover $\Ec\cap B_{R}(X_0).$
\end{lem}
The, quite elementary, proof of this statement in \cite{Mattila} shows also  that if the set $\Ec$ is compact, the system $\Zc$ can be constructed in such way that \emph{its upper estimate in \eqref{mcontrol}} services all points $X_0$ of the set $\Ec$ simultaneously--it is only the upper estimate that is needed in our considerations.  It follows also that the balls $B_{6r}(Z_l))$ form  a covering of the $r$-neighbourhood $\Ec_r=\cup_{X\in \Ec}B_r(X)$ of the set $\Ec.$
\subsection{The Green function and weak solutions}\label{Sect.Green}
\subsubsection{Elliptic equations} In the paper we consider general second order elliptic equations. Unlike the case of analytic or harmonic approximation, the question about the behavior of fundamental solutions at infinity is rather complicated. However for our needs, it suffices to consider the operator in a nice bounded domain $\Omb;$ we accept as $\Omb$  the sufficiently large ball and use the Green function instead.

We consider the  second order formally self-adjoint divergence form  elliptic operator in $\Omb,$
\begin{equation*}
\Lc=-\sum_{j,k}\partial_j a_{j,k}(X)\partial_k.
\end{equation*}
About the matrix of coefficients
\begin{equation*}
  \aF(X)=(a_{j,k}(X))
\end{equation*}
we suppose that it is real, symmetric, continuous, moreover, it is Dini continuous, namely, for $X,Y\in \Omb,$ for a certain continuity modulus $\om_0(t)$ such that
\begin{equation*}
  \om_0(2t)\le A\om_0(t); \, \int_{0}\frac{\om_0(t)}{t}dt<\infty,
\end{equation*}
it satisfies
\begin{equation}\label{Dini}
  |a_{j,k}(X)-a_{j,k}(Y)|\le \om_0(|X-Y|);
\end{equation}
 finally, it is also uniformly elliptic, $\aF(X)\ge\ka>0$ in $\Omb.$

With this operator, we associate the \emph{sesquilinear form}
\begin{equation}\label{sesq.form}
  \ab[u,v]=\ab_{\Omb}[u,v]=\int_{\Omb}\langle\aF(X) \nabla u, \nabla v\rangle dX,
\end{equation}
for $u,v$ in the Sobolev space $H^1(\Omb),$ where (and further on) the angle brackets denote the scalar product in $\C^m$ and the symbol $dX$ is used for the  integration with respect to the $m$-dimensional Lebesgue measure.

For a  measure $\m$ in $\Omb$ with locally finite variation, a function $u$ with locally integrable derivative is called the \emph{weak solution} of the equation $\Lc u=\mu$ if
\begin{equation}\label{weak}
  \ab[u,\psi]=\int_{\Omb}\bar{\psi}(X)d\mu(X)
\end{equation}
for all $\psi\in C_0^\infty(\Om).$

Of course, if the coefficients of the operator $\Lc$ are more regular, say, belong to $C^1,$ and the measure $\mu $ is absolutely continuous with respect to the Lebesgue measure,  $d\mu=g dX,$ the weak solution is, in fact the classical strong solution of the equation $\Lc u=g$ (see, e.g., \cite{Miranda}.)
\subsubsection{The Green function}\label{GreenFunction}
An important role in our construction will be played by the Green function. The existence and main properties of the Green function in our setting are established in the classical papers \cite{LSW}, \cite{GW},  see also \cite{Miranda}.

\begin{thm}\label{Widman}\emph{(Theorem 1.1 in  \cite{GW}).}
  Let $\Lc=-\sum_{j,k}\partial_j a_{j,k}\partial_k$ be an elliptic operator in $\Omb\subset\R^m, \, m\ge 3,$ with the matrix of coefficients, $\aF(X)=(a_{j,k}(X)),$ being bounded and uniformly positive, $\aF,\, \aF^{-1}\in L_\infty(\Omb).$ Then there exists  the  unique Green function $G(X,Y)\ge0$ such that
  \begin{equation}\label{Green}
    \ab[\f(.),G(X,.)]=\f(X), \f\in C_0^{\infty}(\Omb),
  \end{equation}
  where $\ab$ is the sesquilinear form of $\Lc$ see \eqref{sesq.form}; moreover,

  \begin{equation}\label{upper}
    G(X,Y)\le K_1|X-Y|^{2-m}, \, X,Y\in\Omb,
  \end{equation}
  and for $|X-Y|<\frac12\dist(X,\partial\Omb),$
  \begin{equation}\label{lower}
    G(X,Y)\ge K_2|X-Y|^{2-m},
  \end{equation}
  where the constants $K_1,K_2$ depend only on $\|\aF\|_{L_{\infty}}, \|\aF^{-1}\|_{L_{\infty}}$ (and $\Omb$.)
\end{thm}

Additionally, if the coefficients of $\Lc$ are Dini continuous, see \eqref{Dini}, the following  estimate holds (see Theorem 3.3 in \cite{GW}):
\begin{thm} Suppose that the boundary of $ \Omb$ is sufficiently regular. Then 
\begin{equation}\label{deriv.1G}
 |\nabla_X G(X,Y)\le C|X-Y|^{1-m},\,  |\nabla_Y G(X,Y)\le C|X-Y|^{1-m},
\end{equation}
  \begin{equation}\label{deriv.2G}
    |\nabla_X\nabla_Y G(X,Y)|\le C |X-Y|^{-m},
  \end{equation}
  for all $X,Y\in\Omb.$
\end{thm}

We will use these results for the case when $\Omb$ is a (sufficiently large) ball containing the compact set $\Kb$ under consideration, therefore, the regularity conditions are surely satisfied.

The Green function $G(X,Y)$ produces the solution of the Dirichlet  problem:
for a function $f\in L_2(\Omb),$ the scalar product
\begin{equation*}
  u(X)=\int_{\Om}G(X,Y)f(Y) dY\equiv \langle f(.), G(X,.)\rangle
\end{equation*}
defines the (unique) weak solution in $\overset{\circ}{H}{ }^1$ of the equation
$\Lc u=f.$ Moreover,
by Theorem 6.1 in \cite{LSW}, for any measure $\mu$ on $\Omb,$ having bounded variation, the integral
\begin{equation*}
  u(X)=(G(X,.),\mu(.))\equiv \int_{\Omb}G(X,Y)d\mu(Y)
\end{equation*}
represents a weak solution of the equation $\Lc u=\mu,$ vanishing on $\partial\Omb.$


The Green function will  be used in the analysis of the integral representations related to the operator $\Lc$ and in constructing the required approximation.
The standard integral representation is valid for an elliptic equation for a domain $\Om'\subset\Omb,$ see (see, e.g.,  (9.3) in \cite{Miranda}):

\begin{equation}\label{int.repr}
  u(X)=\ab[u(.),G(X,.)]-\int_{\partial\Om'}G(X,Y)\partial_{\n_{\aF}(Y)}u(Y) d\s(Y).
\end{equation}
 This formula  involves the boundary term, usually  called the single layer potential, containing the conormal derivative of $u$ on a surface $\G,$ the boundary of $\Om'$
\begin{equation}\label{conormal}
\partial_{\n_\aF(X)}u(X):=\sum_{j,k}a_{j,k}(X)\n_j(X)\partial_k u(X)=\langle\aF(X)\nabla u(X),\n(X)\rangle, \, X\in \G,
\end{equation}
 where, recall, the angle parentheses denote the standard scalar product in $\C^{m}$ and $\n_j(X)$ are components of the unit normal vector $\n(X)$ of $\G$.
It is possible that a part  of the  boundary of $\Om'$ lies on the boundary of $\Omb;$ in this case the integral over the corresponding part of $\G$ in \eqref{int.repr} vanishes.

\section{The constructive description. Sufficient conditions}\label{Sect.Sufficient}
Further on, we consider the space $C^{\om}(\Kb)$  of continuous functions on the compact set $\Kb$ with continuity modulus, a concave function $\om(t),$ with decaying $\frac{\om(t)}{t},$   satisfying
\begin{equation}\label{omega}
  \int_0^t\frac{\om(\ta)}{\ta}d\ta\le C'\om(\ta),\, \int_t^\infty\frac{\om(\ta)}{\ta^2}d\ta\le C''\frac{\om(t)}{t},
\end{equation}
for some constants $C',C''.$ 
\begin{thm}\label{thm.suff}
Let $\Kb\subset\R^m$ be a compact set with zero Lebesgue measure and let $\om(t)$ be a continuity modulus satisfying \eqref{omega}. For $\de>0,$ denote by $\Kb_\de$ the $\de$-neighborhood of $\Kb$ in $\R^m.$ Suppose that for a function $f(X),$ $X\in\Kb,$ for any  $\de>0$ (equivalently, for a sequence of such $\de$ tending to zero), there exists a function $v_\de$ on $\Kb_\de$ such that
\begin{equation}\label{cond1}
|f(X)-v_\de(X)|\le c_1\om(\de), \, X\in\Kb,
\end{equation}
and
\begin{equation}\label{cond2}
  |\nabla v_\de(X)|\le c_2\om(\de)\de^{-1},\, X\in\Kb_\de,
\end{equation}
with some constants $c_1,c_2$ \emph{(}depending, generally, on the function $f$\emph{)}. Then $f\in C^{\om}(\Kb).$
\end{thm}
\begin{proof}Recall that  the open ball with radius $\de$ and center $X$ is denoted  by $B_\de(X).$ For a given $\de>0,$ we take two arbitrary points  $X_1,X_2\in\Kb$ such that $|X_1-X_2|< \frac{\de}{2}.$ Then both balls $B_{\de}(X_j)$ lie inside $\Kb_\de$ and, therefore, the straight segment connecting $X_1$ and $X_2$ lies inside $\Kb_\de$ as well.  Let $\pmb{\n}=\frac{X_1-X_2}{|X_1-X_2|}.$ Then
\begin{gather*}
  |f(X_2)-f(X_1)|\le |f(X_2)-v_\de(X_2)|+|v_\de(X_2)-v_\de(X_1)|+|v_\de(X_1)-f(X_1)|\\\nonumber \le 2 c_1 \om(\de)+\frac{\de}{2}\int_0^1|\nabla v_\de(X_1+\frac{\de}{2}t\pmb{\n})|dt\le
  \\\nonumber 2 c_1 \om(\de)+c_2\frac{\de}{2}\frac{\om(\de)}{\de}\le C\om(\frac{\de}{2}).
\end{gather*}
\end{proof}
We can see that there are no requirements that the approximating functions $v_\de$ are solutions to some equation or whatever else. The remaining part of the paper is devoted to constructing, for a given function $f\in C^{\om}(\Kb)$ and an arbitrary $\de$ such approximating functions, satisfying \eqref{cond1}, \eqref{cond2}, which  are weak solutions of a given second order elliptic equation.

\section{The extension of a continuous function}\label{Sect.Ext}

It is  known since long ago, see \cite{McShane}, \cite{Whitney}, that a continuous function on a compact set in the Euclidean space, with a given continuity modulus $\om(t)$, admits an extension to the whole space with the same continuity modulus. The following statement provides us with an extension which, additionally, is infinitely differentiable outside $\Kb$, with  the first order derivative  controlled by this   continuity modulus  and by  the distance to the compact. The continuity modulus $\om(t)$ is fixed, and it is supposed to be a concave function with decaying $\frac{\om(t)}{t},$ satisfying \eqref{omega}

In the following reasoning (and further on in the paper), the symbol $c$ without subscript  denotes constants whose value is of no importance, somehow depending on the function $f$ and the set $\Kb$, which may vary from one formula to another one and even within one and the same formula.
For a point $X\in\R^m\setminus\Kb$ we denote  by $d(X)$ the distance of $X$ to $\Kb.$
\begin{lem}\label{lem.ext}
Let $\Kb\subset\R^m,$ $m\ge3,$ be a compact set with connected complement in $\Omb$. Let $f$ be a continuous function, $f\in C^\om(\Kb)$ with continuity modulus $\om.$ Then there exists an extension $f_0$ of $f$, $f_0\in C_0(\R^m)\cap C^\infty(\R^m\setminus \Kb)$, such that
\begin{equation}\label{Ext1}
|f_0(X_1)-f_0(X_2)\le c\om(|X_1-X_2|); \, X_1,X_2\notin\Kb
\end{equation}
\begin{equation}\label{Ext2}
|\nabla f_0(X)|\le c\frac{\om(d(X))}{d(X)}, \, X\notin\Kb;
\end{equation}
\end{lem}
\begin{rem} In fact, for any $k>0,$ the estimate $\nabla^kf_0(X)=O(\frac{\om(d(X))}{d(X)^k})$ holds.
\end{rem}
\begin{proof}The construction of $f_0$ goes in the following way. Let $f_1\in C_0(\R^m)$ be the McShane-Whitney continuation of $f$ (see \cite{McShane}, \cite{Whitney}) such that
\begin{equation*}
  |f_1(X_1)-f_1(X_2)|\le c \om(|X_1-X_2|), \, X_1,X_2\in\R^m.
\end{equation*}
There exists (see, e.g. \cite{stein}) a smooth approximation $d_0(X)>0,$ $d_0\in C^{\infty}(\R^m\setminus \Kb),$  of the distance function $d(X)$,  such that
\begin{equation}\label{Ext4}
  c_1 d(X)\le d_0(X)\le c_2 d(X), \, X\in \R^m\setminus\Kb,
\end{equation}
and also
\begin{equation*}
  |\grad d_0(X)|\le c,\, |\grad^2 d_0(X)|\le\frac{c}{d(X)},
\end{equation*}with constant $c$ depending only on the compact $\Kb.$

We set $\ro(X)=\frac{d_0(X)}{2c_2},$ where $c_2$ is the constant in \eqref{Ext4}, so that $\ro(X)\le \frac{d(X)}2,$
and denote   $\f(X)=\ro(X)^{-2},\, X\notin\Kb;$ obviously, \begin{equation}\label{fi.est}
|\nabla\f(X)|\le c d(X)^{-3}.
\end{equation}
We also set $\p(X)=\f(X)^{m/2}=\ro(X)^{-m};$ for the derivatives of $\p,$ we have
\begin{equation}\label{psi.est}
  |\nabla \p(X)|\le c d(X)^{-m-1}.
\end{equation}

The  function whose existence is  proclaimed in Lemma \ref{lem.ext}  is constructed as  the averaging  of the function $f_1$ with variable radius depending on $\ro(X).$

We denote by $\Bb(X)$ the ball $B_{\ro(X)}(X),$ centered at the point $X\notin\Kb$ and with radius $\ro(X);$ this ball satisfies $\dist(\Bb(X), \Kb)\ge d(X)/2$. Let  the even function $\hb\ge 0,$ $\hb(r)\in C^\infty_0((-1,1))$, satisfy
\begin{equation*}
  \int_{\R^m}\hb(|X|^2)dX=1.
\end{equation*}
Now we introduce the mollification kernel $\Kc(X,Y),$ smooth in both variables, by
\begin{equation*}
  \Kc(X,Y)=\p(X)\hb(\f(X)|X-Y|^2),
\end{equation*}
thus,
\begin{equation*}
   \int_{\R^m}\Kc(X,Y)dY \equiv \int_{\Bb(X)}\Kc(X,Y)dY=1.
\end{equation*}

Using this kernel, we introduce the mollified  function $f_0$ by setting
\begin{equation*}
  f_0(X)=\int_{\R^m}\Kc(X,Y)f_1(Y)dY\equiv\int_{\Bb(X)}\Kc(X,Y)f_1(Y)dY,
\end{equation*}
 so, $f_0$ is a variable radius mollification of the  function $f_1$ with radius vanishing as the point approaches $\Kb$ (in a somewhat different  setting, the idea of using  a variable radius mollification was earlier implemented by E. Dyn'kin, \cite{Dyn76}, and V. Burenkov, \cite{Bur}). The smoothness of  $f_1$ follows from the smoothness of the kernel $\Kc.$

For a given point $X\notin \Kb,$ let $X_0$ be (one of) the point(s) in $\Kb,$ closest to $X,$ $d(X)=|X-X_0|.$
 Then
 \begin{equation}\label{diff.1}
   f(X)-f(X_0)=\int_{\Bb(X)}\Kc(X,Y)(f_1(Y)-f(X_0))dY.
 \end{equation}
Since for $Y\in \Bb(X)$ we have $|Y-X_0|\le |Y-X|+|X-X_0|\le c d(X),$ it follows
\begin{equation}\label{diff.2}
  |f_1(Y)-f(X_0)|\le c\om(d(X)),
\end{equation}
in particular, for $Y=X,$ we have
\begin{equation}\label{diff.3}
  |f_1(X)-f(X_0)|\le c \om(d(X)).
\end{equation}
This implies that $f_0$ is, in fact,  a continuous extension of the function $f$ to $\R^m.$

Our aim now is to estimate the derivatives of $f_0.$

For any constant $\cb,$  the derivatives of $f_0$ and of $f_0^{\cb}:= f_0-\cb$, of course, coincide. Therefore, when estimating derivatives of $f_0$ at the point $X\in\R^m\setminus\Kb,$  we may handle the function $f_0^{\cb}$  instead,   using the value $\cb=f(X_0)$ (where, recall, $X_0$ is the closest to $X$ point in $\Kb$). We keep in mind here that, by \eqref{diff.2} and \eqref{diff.3},
\begin{equation}\label{fcb}
  |f_0^{\cb}(Y)|\le c \om(d(X)), \,\mbox{for}\, Y\in \Bb(X).
\end{equation}

So, we have
\begin{equation*}
\nabla f_0^{\cb}(X) =\int_{\R^m}\nabla_X\Kc(X,Y)f_1^{\cb}(Y)dY
=\int_{\Bb(X)}\nabla_X\Kc(X,Y) f_1^{\cb}(Y)dY.
\end{equation*}
Since $|f_1^{\cb}(Y)|=|f_1(Y)-f(X_0)|\le c \om(d(X))$ for $Y\in \Bb(X)$,
\begin{equation*}
|\nabla f_0(X)|\le c\om(d(X))\int_{\Bb(X)}\left|\nabla_X\Kc(X,Y)\right|dY.
\end{equation*}
For the derivatives of the kernel $\Kc,$ we have
\begin{gather}\label{deriv.3}
 \nabla_X\Kc(X,Y)=\\\nonumber
  (\nabla\p(X))\hb(\f(X)|X-Y|^2)+\p(X)\nabla_X\hb(\f(X)|X-Y|^2);
 \end{gather}
 where
 \begin{equation}\label{deriv.3.0}
  \nabla_X\hb(\f(X)|X-Y|^2)= \hb'(\f(X)|X-Y|^2)
  \left(\nabla\f(X)|X-Y|^2+2\f(X)(X-Y)\right).
 \end{equation}
 Using \eqref{fi.est}, \eqref{psi.est}, and the fact that $|X-Y|\le \ro(X)$, we obtain from \eqref{deriv.3}, \eqref{deriv.3.0} that
 \begin{equation*}
   |\nabla_X \Kc(X,Y)|\le c d(X)^{-m-1},
 \end{equation*}
 and, therefore,
 \begin{equation*}
   \int_{\Bb(X)}|\nabla_X \Kc(X,Y)|\le c d(X)^{-1}.
 \end{equation*}
 We recall \eqref{fcb}, which gives us
 \begin{equation*}
  |\nabla f_0(X)| \equiv |\nabla f_0^{\cb}(X)|\le c \frac{\om(d(X))}{d(X)}.
 \end{equation*}
\end{proof}
Finally, we can take a smooth function $\chi(X)\in C^{\infty}_0(\Omb), $
$\chi(X)=1$ in a neighborhood of $\Kb$ and consider $\chi  f_0$ instead of $f_0,$ to have a continuation of $f$ with compact support in $\Omb,$ with all the required properties.

\section{Integral estimates for  the continuity modulus}\label{Sect.Estimates} In this section we derive important estimates of some integrals involving the continuity modulus $\om(t)$ and the extended function $f_0$ constructed in Sect. \ref{Sect.Ext}. Let, as before, $\Omb$ be an open ball in $\R^m,$ such that $\Kb\subset \Omb.$ Suppose that $\Lc$ is an elliptic operator in $\Omb$ with coefficients satisfying the conditions in Sect. \ref{Sect.Green} and let $G(X,Y)$ be the Green function for $\Lc$ in $\Omb.$

We suppose that our compact set $\Kb$ satisfies the following condition.

 \textbf{Condition A.} For any $\de>0$  there exists a set $\Zc(\de)$ of points $Z_l\in \Kb,$ $l=1,\dots, M(\de),$ such that the distance between any two points in $\Zc(\de)$ is not less than $\de,$ additionally,  for some constant $c_0,$ the set $\Kb_\de,$ the $\de$-neighborhood of $\Kb$ is covered by the collection of closed balls $B_{c_0\de}(Z_l),$
\begin{equation}\label{cover}
  \Kb_\de\subset\bigcup_{Z_l\in \Zc(\de)}B_{c_0\de}(Z_l)
\end{equation}
and, finally, for any $R\ge\de,$ some constant $c_1$ and for any point $\Xb\in\R^m$, the quantity $M(\Xb,R,\de)$ of the points of $\Zc(\de)$ in the ball $\overline{B_R(\Xb)}$ is not greater than $c_1\left(\frac{R}{\de}\right)^{m-2},$
in particular, the number  $M(\de)$ of balls in $\Zc(\de)$ satisfies
\begin{equation}\label{number of points}
M(\de)
  \le c_1\left(\frac{\wr \Kb\wr}{\de}\right)^{m-2}.
\end{equation}

As follows from Lemma \ref{Lem.cover}, if the set $\Kb$ is $(m-2)$-regular, condition \textbf{A} is satisfied. As the visual example of such sets $\Kb$, one can consider a compact Lipschitz surface with codimension $2$ in $\R^m,$ as well as a finite union of such surfaces.

Having the set $\Zc(\de),$ for a fixed  $c>0,$ we denote by $\Bc_c(\de)$ the union of balls with radius $c\de$ centered at the points $Z_l\in\Zc(\de),$
\begin{equation}\label{Bc(de)}
  \Bc_c(\de)=\bigcup_{Z_l\in\Zc(\de)}B_{c\de}(Z_l).
\end{equation}
 The value of $c$ in \eqref{Bc(de)} will be fixed so that $c>\max(2c_0,c_1),$ where $c_0,c_1$ are the constants in Condition \textbf{A},  and  we will omit it in the notation, writing simply $\Bc(\de):=\Bc_{c}(\de).$ With this choice,  the set $\Bc(\de)$ covers $\Kb_\de.$

Let $f(X)$ be a function on $\Kb$ with the continuity modulus $\om$, i.e., $f\in C^{\om}(\Kb),$ and let  $f_0$ be its extension to $\Omb$ as constructed in Section \ref{Sect.Ext}, in particular, supported in $\Omb.$

For a fixed  point $X\in(\Omb\setminus\Kb),$  we  denote   the distance from $X $ to $\Kb$ by $\de=d(X)$ and  let $\de_1$ satisfy
$2c_0\de_1<\de,$ where $c_0$ is the constant in \eqref{cover}.

We introduce the set $\Theta_{\de_1}\subset \Omb,$
\begin{equation}\label{Theta}
  \Theta_{\de_1}=\Omb\setminus\overline{\Bc(\de_1)}.
\end{equation}
It follows that the initial point $X$ belongs to $ \Theta_{\de_1},$ this means, it lies outside $\overline{\Bc(\de_1)}.$   The distance of the points in  $\Theta_{\de_1}$  to $\Kb$ can be estimated from below,
\begin{equation}\label{d(X1)}
  \dist(\Theta_{\de_1},\Kb )\ge c_0\de_1.
\end{equation}
In fact, to justify \eqref{d(X1)},  we suppose that $\dist(\Kb,\Theta_{\de_1} )< c_0\de_1 $ and consider a point $\Xb$ in the joint boundary of $\Theta_{\de_1}$ and $\Bc(\de_1)$ where \eqref{d(X1)} breaks down. Let $\Xb'$ be  its closest point in $\Kb,$ so that  $\dist(\Xb,\Kb)\equiv d(\Xb)=|\Xb-\Xb'|.$
The point $\Xb'$ must belong to one of the balls $B_{c_0\de_1}(Z_l),$ $Z_l\in \Zc(\de_1).$
Therefore,  for this particular $l,$
\begin{equation*}
  |Z_{l}-\Xb|\le |Z_l-\Xb'|+|\Xb'-\Xb|< c_0\de_1+c_0\de_1=2c_0\de_1.
\end{equation*}
This would means that $\Xb$ is an interior point in $\Bc(\de_1)$ and not a boundary point, contrary to what  was assumed.

The boundary $\Si(\de_1):=\partial(\Bc(\de_1))$ of $\Bc(\de_1)$ is a piecewise smooth surface; it consists of a finite set of domains on spheres in $\R^m,$ having piecewise smooth $m-2$-dimensional boundaries.
We denote by $\s$ the $m-1$ - dimensional surface measure on $\Si({\de_1}) $, generated by the Lebesgue measure on $\R^m$.  Since the surface measure of a sphere equals  $\s(\partial( B_{c\de_1}(Z)))=c_{m}\de_1^{m-1}$ and the quantity $M(\de_1)$  of balls satisfies \eqref{number of points}, we have
\begin{equation}\label{Surface}
  \s(\partial \Bc(\de_1))\le c M(\de_1)\de_1^{m-1}\le c\wr\Kb\wr^{m-2}\de_1^{2-m}\times \de_1^{m-1}=c\de_1.
\end{equation}

Further on, when estimating the expressions in the integral representations to follow, we will use the following lemma.

\begin{lem}\label{lem.modulus.int} For the continuity modulus $\om$ as above, the following estimate holds
  \begin{equation}\label{est.mod}
    \int\limits_{\Bc(\de_1)}\frac{\om(d(Y))}{d(Y)}dY<\infty, \, \de_1>0.
  \end{equation}
\end{lem}
\begin{proof} First of all, since the set $\Kb$ has Lebesgue measure zero, we may replace the integration domain in \eqref{est.mod} by $\Bc(\de_1)\setminus \Kb$. We introduce the notation $\de_{n}=2^{1-n}\de_1,$ $n=1,2,\dots;$ for each $n$,  consider the covering of $\Kb$ by the balls with radius $c_1\de_n,$ in accordance with Condition\textbf{ A} (in particular, for an $(m-2)$-regular set $\Kb$, using Lemma \ref{Lem.cover}), and construct  the corresponding neighborhoods  $\Bc(\de_n)$ of the compact set $\Kb$.   Accordingly, the domain $\Bc(\de_1)\setminus\Kb$ admits the representation
\begin{equation*}
  \Bc(\de_1)\setminus\Kb=\bigcup_{n=1}^\infty \Gc_n
\end{equation*}
where $\Gc_n$ denotes $\Bc(\de_n)\setminus \Bc(\de_{n+1}).$ Therefore,
\begin{equation}\label{Int.strip}
  \int\limits_{\Bc(\de_1)\setminus \Kb}\frac{\om(d(Y))}{d(Y)}dY=
\sum_{n=1}^\infty\int_{\Gc_n}\frac{\om(d(Y))}{d(Y)}dY.
\end{equation}
On the set $\Gc_n,$ we have $d(Y)\ge c \de_{n}$ and $\frac{\om(d(Y))}{d(Y)}\le C \frac{\om(\de_{n+1})}{\de_{n+1}}.$ The quantity $M(\de_n)$ of balls composing $\Bc(\de_n)$, by Condition \textbf{A}, is majorated by $\left(\frac{ \wr\Kb\wr}{\de_n}\right)^{m-2},$ therefore,
\begin{equation*}
\meas_m(\Bc(\de_n))\le c \de_n^{m}\left(\frac{ \wr\Kb\wr}{\de_n}\right)^{m-2}\le c \de_n^2,
\end{equation*}
 which implies
\begin{equation*}
  \int_{\Gc_n}\frac{\om(d(Y))}{d(Y)}dY\le c\frac{\om(\de_{n+1})}{\de_{n+1}} \meas_m(\Gc_n)\le c\om(\de_{n+1})\de_n^2.
\end{equation*}
We  substitute the last estimate in \eqref{Int.strip} and obtain
\begin{equation*}
  \int_{\Bc(\de)}\frac{\om(d(Y))}{d(Y)}dY \le \sum_n \frac{\de_1}{2^n} \om\left(\frac{\de_1}{2^n}\right)<\infty.
\end{equation*}
\end{proof}
Lemma \ref{lem.modulus.int} implies the following important property.
\begin{cor}\label{Cor.L1} As $\de_1\to +0,$ the integral in \eqref{est.mod} tends to zero.
\end{cor}

We will also need the following localized estimate.
\begin{lem}\label{lem2} Let  $X_0$ be an arbitrary point in $\Kb$ and let $r\ge\de.$ Then
\begin{equation}\label{L.2.estimate}
\Ib(X_0,r):=  \int\limits_{B_r(X_0)\cap\Bc(\de_1)}
  \frac{\om(d(Y))}{d(Y)} dY\le cr^{m-2}\de_1\om(\de_1).
\end{equation}
\end{lem}
\begin{proof}
  We use the notations for $\de_n,\Gc_n$ from the previous proof.  Then
  \begin{equation}\label{IB1}
   \Ib(X_0,r) =\sum_{n=1}^\infty\int\limits_{B_r(X_0)\cap \Gc_{n}}\frac{\om(d(Y))}{d(Y)}dY.
  \end{equation}
 For $Y\in B_r(X_0)\cap\Gc_n$, we have
 \begin{equation*}
   d(Y)\ge c_1\de_{n+1}, \, \frac{\om(d(Y))}{d(Y)}\le c\frac{\om(\de_{n+1})}{\de_{n+1}}.
 \end{equation*}
 The ball $B_{r+2c_1\de_n}(X_0)$ contains no more than
 \begin{equation*}
  \left(\frac{r+2c_1\de_n}{\de_n}\right)^{m-2}\le c\left(\frac{r}{\de_n}\right)^{m-2}
\end{equation*}
points of the set $\Zc(\de_1)$, therefore,
 \begin{equation*}
   \meas_{m}(B_r(X_0)\cap\Bc(\de_n))\le c \de_n^m \left(\frac{r}{\de_n}\right)^{m-2}=c r^{m-2}\de_n^2.
 \end{equation*}
 The latter inequality enables us to estimate the integral in  \eqref{IB1}:
 \begin{equation}\label{IB2}
  \sum_n \int\limits_{B_r(X_0)\cap\Gc_n}
   \frac{\om(d(Y))}{d(Y)}\le c\sum_n\frac{\om(\de_{n+1})}{\de_{n+1}}r^{m-2}\de_{n}^2\le
   c\de_1\om(\de_1) r^{m-2}.
 \end{equation}
\end{proof}

The reasoning in the proof of Lemma \ref{lem2} produces one more important estimate. For $X_0\in\Kb$ and an arbitrary $\de>0$, we have $B_{c_1\de}(X_0)\subset \Bc(\de).$ Therefore,
\begin{equation}\label{L2.a}
  \int_{B_\de(X_0)} \frac{\om(d(Y))}{d(Y)}\le c \de^{m-1}\om(\de).
\end{equation}

\section{The integral representation}\label{Sect.Int.repr} For a fixed function $f$ on $\Kb$ with continuity modulus $\om,$ as above, let $f_0$ be the extension of $f$ to $\R^m$ constructed in accordance with Lemma \ref{lem.ext}. The aim of this section is to justify the integral representation for the function $f_0$ in $\Omb.$ At the moment, we may not  apply immediately the general formula \eqref{int.repr}  since the derivatives of $f_0$ may have rather strong singularities when approaching the set $\Kb.$ Therefore, a limit procedure is needed in order to justify the resulting formula.
\subsection{Integral representation outside $\Kb$}
Recall that   $f_0$ vanishes at the boundary of the ball $\Omb.$
 We fix some point $X\notin \Kb.$ For a  sufficiently small $\de_1>0,$ such that $X\in \Theta_{\de_1}$ (see the definition in \eqref{Theta}), and for  the  function  $f_0$ which by construction is smooth outside $\Kb$, we apply the formula for the integral representation in the domain $\Theta_{\de_1},$ which is possible since $f_0$ is smooth in this domain; we  use the Green function $G(X,Y)$ for the ball $\Omb.$ The boundary of $\Theta_{\de_1}$ consists of two parts. The first one is the boundary of the ball $\Omb.$ Since $G$ vanishes on this sphere, this part of the boundary does not contribute to the integral representation. The other part of the boundary of $ \Theta_{\de_1}$ is the interior part, namely, the joint boundary $\Si(\de_1)$ of  $ \Theta_{\de_1}$ and $\Bc(\de_1),$ the latter being the  neighborhood of the compact $\Kb$. Thus, for $X\in \Theta_{\de_1},$ we have

\begin{gather}\label{representation}
  f_0(X)=\int_{ \Theta_{\de_1}}\langle\aF(Y)\nabla_YG(X,Y),\nabla_Y f_0(Y)\rangle dY +\\\nonumber
  \int_{\Si(\de_1)}G(X,Y)\partial_{\n_\aF(Y)}f_0(Y)d\s(Y).
\end{gather}

When $\de_1$ is sufficiently small, namely when the distance from $X$ to $\Si({\de_1})$ is greater than $\frac12 d(X)$,  the Green function $G(X,Y)$  and its derivative $G(X,Y)$ satisfy  for  $Y\in \Si(\de_1)$ the estimates
\begin{equation*}
  G(X,Y)\le c d(X)^{2-m}, \, |\nabla_Y G(X,Y)|\le c d(X)^{1-m},
\end{equation*}
according to   \eqref{upper}, \eqref{deriv.1G}.

We are able now to estimate the integrals in \eqref{representation}. For the surface integral,  by \eqref{Ext2} and \eqref{Surface}, we have
\begin{gather}\label{Single.1}
  \left|\int_{\Si(\de_1)}G(X,Y)\partial_{\n_{\aF}(Y)}f_0(Y)d\s(Y)\right|\le \\\nonumber c d(X)^{1-m} \int_{\Si(\de_1)}\frac{\om(d(Y))}{d(Y)}dY\le c d(X)^{1-m}\de_1\frac{\om(\de_1)}{\de_1}=c d(X)^{1-m}\om(\de_1).
\end{gather}
Estimate \eqref{Single.1} shows that the surface integral in \eqref{representation} tends to zero as $\de_1\to 0.$

Now we consider the volume integral in \eqref{representation}. With $X\notin\Kb$ still  fixed, we use estimate \eqref{deriv.1G} for the gradient of the Green function and, again,  estimate \eqref{Ext2} for the gradient of the function $f_0.$ Together with  Lemma \ref{lem.modulus.int} and Corollary \ref{Cor.L1},  this implies that as $\de_1\to 0,$ the volume integral tends to
\begin{equation*}
  \int_{\Omb}\langle\aF(Y)\nabla_Y G(X,Y),\nabla f_0(Y)\rangle dY.
\end{equation*}

So,  we may  pass  to the limit as $\de_1\to 0$ in \eqref{representation},  use Corollary \ref{Cor.L1}, and obtain the integral representation of the function $f_0$ \emph{outside} $\Kb$:
\begin{equation}\label{Repr.fin}
  f_0(X)=\int_{\Omb}\langle \aF(Y)\nabla_Y G(X,Y),\nabla f_0(Y)\rangle dY,\, X\notin \Kb.
\end{equation}

Our next task is to show that this representation for $f_0=f$ is valid for   all points \emph{ in}  $\Kb$ as well. For this,  we need the following estimate.
\begin{lem}\label{Lem3} Let $\Kb$ satisfy Condition \textbf{A}
  and let $X_0\in\Kb$. Then for $r>0,$
  \begin{equation}\label{est.Lem3}
 \Ib(X_0,r)  \equiv \int_{B_r(X_0)}\frac{\om(d(Y))}{d(Y)}|Y-X_0|^{1-m}dY\le c\om(r).
  \end{equation}
\end{lem}
\begin{proof} We set $r_n=2^{-n}r,$ $n=0,\dots,$  $\Vc_n=B_{r_n}(X_0)\setminus B_{r_{n+1}}(X_0). $ Then the integral in \eqref{est.Lem3} splits into the sum
\begin{equation*}
  \Ib(X_0,r)=\sum_{n=0}^{\infty}\Jb_n; \, \Jb_n\equiv\int_{\Vc_n}\frac{\om(d(Y))}{d(Y)}|Y-X_0|^{1-m}dY.
\end{equation*}
When $Y$ belongs to the spherical annulus $\Vc_n$, we have $|Y-X_0|\ge r_{n+1},$ and, therefore,
\begin{gather*}
  \Jb\le r^{1-m}_{n+1}\int_{\Vc_n}\frac{\om(d(Y))}{d(Y)}dY\le
r_{n+1}^{1-m}\int_{B_{r_n}}\frac{\om(d(Y))}{d(Y)}dY\\\nonumber
\le c r_{n+1}^{1-m} r_n^{m-1}\om(r_n)\le c \om(2^{-n}r).
\end{gather*}
We take into account the properties of the continuity modulus $\om$ to obtain
\begin{gather*}
  \Ib(X_0,r) \le c \sum_{n=0}^\infty \om(2^{-n}r)\le c \int_0^\infty\om(2^{-t}r)dt\\\nonumber
  =c\int_1^\infty \om(\ta^{-1}r)\frac{d\ta}{\ta}=c\int_0^\ro s^{-1}\om(s) ds\le c \om(r.)
\end{gather*}
\end{proof}
Our  next aim is to determine the behavior of the quantity
\begin{equation}\label{potential.1}
 \Hb(X,\de_0)= \int_{B_{2\de_0}(X)}\frac{\om(d(Y))}{d(Y)}|X-Y|^{1-m}dY
\end{equation}
as the point $X\notin\Kb$ approaches $\Kb,$ while $\de_0=d(X).$ Note that in \eqref{potential.1}, the ball where the integration is performed, although having center $X$ outside $\Kb$, still may contain a certain portion of $\Kb.$ Therefore, for $Y\in\Kb,$ in other words, when $d(Y)=0,$ we set formally $\frac{\om(d(Y))}{d(Y)}=0,$ so that the integral in \eqref{potential.1} makes sense.
\begin{lem}\label{Lem4} For $X\notin\Kb,$
\begin{equation}\label{est.Lem4.1}
\Hb(X,\de_0)\le c \om(\de_0).
\end{equation}
\end{lem}
\begin{proof} We consider two domains in the ball $B_{2\de_0}(X)$, where separate estimates will be performed.

If $Y$ belongs to the ball $B_{\de_0/2}(X),$  we have $d(Y)\ge \frac12 \de_0.$ Therefore, for the integral $\Ib^{(1)}$ over this ball, the estimate holds
\begin{equation*}
\Ib^{(1)}\le \frac{\om(\de_0/2)}{\de_0}\int_{B_{\de_0/2}(X)}|X-Y|^{1-m} dY \le c\om(\de_0/2)\le c \om(\de_0).
\end{equation*}
For the integral  $\Ib^{(2)}$ over the remaining part  of the ball $B_{2\de_0}(X)$, the points $X,Y$ are controllably separated, $|X-Y|\ge \frac12 \de_0.$ Let $X_0\in\Kb$ be a point in $\Kb$ such that $|X-X_0|=\de_0$.  Then the ball $B_{2\de_0}(X)$ is contained in the ball  $B_{4\de_0}(X_0)$, and therefore, by Lemma \ref{Lem3},
\begin{gather*}
  \Ib^{(2)}=\int\limits_{B_{2\de_0}(X)\setminus B_{\de_0/2}(X)}\frac{\om(d(Y))}{d(Y)}|X-Y|^{1-m}dY\\\nonumber
\le c \de_0^{1-m}\int_{B_{4\de_0}(X_0)}\frac{\om(\de(Y))}{d(Y)}dY\\\nonumber
\le c \de_0^{1-m}\de_0^{m-1}\om(2\de_0)\le c \om(\de_0).
\end{gather*}
We sum these estimates for $\Ib^{(1)}$  and $\Ib^{(2)}$ to obtain \eqref{est.Lem4.1}.
\end{proof}

\subsection{Integral representation on $\Kb$}
We define  the function $f_0^*$ on the compact $\Kb$ in the following way: for a point $X_0\in \Kb,$ we set
\begin{equation}\label{f*}
  f_0^*(X_0)=\ab[f_0(.),G(X_0,.)]\equiv \int_{\Omb}\langle  \nabla f_0(Y), \aF(Y)\nabla_YG(X_0,Y)\rangle dY.
\end{equation}
\begin{lem}\label{Thm.representation}
  Let $X\notin\Kb,$ $X_0\in\Kb,$ $|X-X_0|=d(X)=\de.$
  Then
  \begin{equation}\label{convergence}
    |f_0^*(X_0)-f_0(X)|\le c \om(\de).
  \end{equation}
 Similarly, for $X_0,X_1\in \Kb,$ $|X_0-X_1|=\de,$
  \begin{equation}\label{convergence inside}
    |f_0^*(X_0)-f_0^*(X_1)|\le c \om(\de).
  \end{equation}
\end{lem}
\begin{proof} We present the proof of \eqref{convergence}. The reasoning establishing  \eqref{convergence inside} is quite similar.

 We set $r_n=2^{n+1}\de,$ $n\ge1.$ Since the function $f_0$ is smooth outside $\Kb$, at the points $X\notin\Kb$ the integral representation \eqref{representation}  is valid. The aim of this lemma consists in establishing that this integral representation is valid for all points in $\Kb$ as well.  So, we have

\begin{gather}\label{difference}
  f_0^*(X_0)-f_0(X) =\\\nonumber \int\limits_{\Omb}\langle\aF(Y)\nabla f_0(Y),\nabla_YG(X_0,Y)\rangle dY-  \int\limits_{\Omb}\langle\aF(Y)\nabla f_0(Y),\nabla_YG(X,Y)\rangle dY=
  \\\nonumber
  \int\limits_{B_{r_0}(X_0)} \langle\aF(Y)\nabla f_0(Y),\nabla_Y(G(X_0,Y))\rangle dY-\int\limits_{B_{r_0}(X_0)}\langle \aF (Y)\nabla f_0(Y),\nabla_YG(X,Y)\rangle dY+\\\nonumber
  \sum_{n=1}^\infty \int\limits_{B_{r_n}(X_0)\setminus B_{r_{n-1}}(X_0)}\langle \aF(Y)\nabla f_0(Y),\nabla_Y(G(X_0,Y)-G(X,Y))\rangle dY\\\nonumber
  \equiv (I_0(X_0)-I_0(X))+\sum_{n=1}^\infty (I_n(X_0)-I_n(X)).
\end{gather}
We recall the estimates \eqref{deriv.1G} for the Green function $G(X,Y):$
\begin{equation*}
  |\nabla_YG(X,Y)|\le c|X-Y|^{1-m}, \,  |\nabla_YG(X_0,Y)|\le c|X_0-Y|^{1-m}.
\end{equation*}
Therefore, for the first term on the last line in \eqref{difference},  we have, by Lemma \ref{Lem3} and Lemma \ref{Lem4},
\begin{gather}\label{I0}
  I_0(X_0)\le c \int_{B_{r_0}(X_0)}|Y-X_0|^{1-m}|\nabla f_0(Y)|dY\le\\\nonumber
  c\int\limits_{B_{r_0}(X_0)}|Y-X_0|^{1-m}\frac{\om(d(Y))}{d(Y)}dY\le c\om(r_0)\le c\om(\de), \,\mbox{and,}\, \mbox{similarly,}\\\nonumber
  I_0(X)\le c\int{B_{r_0}(X_0)}|Y-X_0|^{1-m}|\nabla f_0(Y)|dY\le \\\nonumber c\int\limits_{B_{2r_0}(X)}|Y-X|^{1-m}\frac{\om(d(Y))}{d(Y)}dY\le
  c\om(\de).
\end{gather}
To handle  the sum on the last line in \eqref{difference}, we use the estimate \eqref{deriv.2G} for the Green function, which implies
\begin{equation*}
 | \nabla_Y G(X_0,Y)-\nabla_YG(X,Y)|\le c|X_0-X|\max_{\ta\in[0,1]}|Y-(\ta X+(1-\ta)X_0)|^{-m}.
\end{equation*}
For $Y\notin B_{r_n}(X_0)$, $n>1$, we have $|Y-X_0|\ge 2 \de=|X-X_0|,$  therefore,
\begin{equation*}
  |Y-(\ta X+(1-\ta)X_0)|\ge |Y-X_0|-(1-\ta)|X-X_0|\ge |Y-X_0|-\de\ge \frac12 |X-X_0|,
\end{equation*}
which implies
\begin{equation*}
   | \nabla_Y G(X_0,Y)-\nabla_YG(X,Y)|\le c\de |Y-X_0|^{-m}.
\end{equation*}
As a result,
\begin{gather*}
  |I_n(X_0)-I_n(X))|\le c\de \int\limits_{B_{r_n}(X_0)\setminus B_{r_{n-1}}(X_0)}|Y-X_0|^{-m}\frac{\om(d(Y))}{d(Y)}dY\le \\\nonumber
  c\de (2^n\de)^{-m}\frac{\om(2^n\de)}{2^n\de}\int_{B_{r_n}}dY=c\de(2^n\de)^{-m}\frac{\om(2^n\de)}{2^n\de}\times(2^n\de)^m=c2^{-n}\om(2^n\de)
\end{gather*}
Therefore, for the sum in  \eqref{difference}, we have
\begin{gather}\label{Sum.In}
  \sum_{1}^{\infty}(I_n(X_0)-I_n(X))\le c\sum_{1}^{\infty}2^{-n}\om(2^n\de)\le c\int_0^{\infty}\frac{\om(2^t\de)}{2^t}dt=\\\nonumber
  c\int_1^\infty\frac{\om(\de\vs)}{\vs^2}d\vs =c\de\int_{\de}^\infty
\frac{\om(\uu)}{\uu^2}d\uu\le c \om(\de),\end{gather}
by \eqref{omega}.

The substitution of \eqref{I0} and \eqref{Sum.In} into \eqref{difference} gives \eqref{convergence} and thus proves the Lemma.
\end{proof}

\begin{lem}\label{Thm.coincidence}
  Let $X_0$ be a point $\Kb.$ Let $f_0^*(X_0)$ be constructed as in \eqref{f*}. Then
  \begin{equation*}
    f_0^*(X_0)=f(X_0).
  \end{equation*}
\end{lem}
\begin{proof}We take an arbitrary small $\de>0$ and find a point $X\notin \Kb$ such that $\frac{\de}2\le |X-X_0|\le\de$ (such point, of course, exists since  $\Kb$ possesses no interior points.) Let, further, $X_1\in\Kb$ be (one of) the closest point(s) in $\Kb$ to $X$ and we set $d(X)=|X_1-X|$. Then $d(X)\le \de$ and $|X_0-X_1|\le |X_0-X|+|X_1-X|\le 2\de.$ By Theorem \ref{Thm.representation},
\begin{equation*}
  |f_0^*(X_0)-f_0(X)|\le |f_0^*(X_0))-f^*_0(X_1)|+|f_0^*(X_1)-f_0(X)|\le c\om(\de).
\end{equation*}
  We recall now  that the  function $f_0$ satisfies $|f(X_0)-f_0(X)|\le c \om(\de),$ according to the way how  $f_0$ was constructed.  Therefore,
  \begin{equation*}
    |f_0^*(X_0)-f(X_0)|\le |f_0^*(X_0)-f_0(X)|+|f_0(X)-f(X_0)|\le c\om(\de).
  \end{equation*}
  Since $\de>0$ was chosen arbitrarily, it follows that $f_0^*(X_0)-f(X_0)=0.$
\end{proof}

Thus, we have established that the integral representation \eqref{Repr.fin} is valid not only for $X\notin\Kb$ but for $X\in\Kb$ as well, this means on the whole domain $\Omb.$

\section{Construction of the approximating  function}\label{Sect.solution}
\subsection{Formulation of main theorem}
The aim of this section is to prove the approximation theorem.
\begin{thm}\label{Main.Theorem}Let $\Kb\subset\R^m,$ $m\ge 3,$ be a compact set satisfying Condition \textbf{A}. Let $\Lc$ be an elliptic operator in the neighborhood of $\Kb,$ satisfying conditions of Sect. \ref{GreenFunction}. Suppose that the function $f(X),$ $X\in\Kb,$ belongs to $C^{\om}(\Kb),$ where $\om$ is a continuity modulus as above. Then for some constants $c_1,c_2,c_3>0,$ for any $\de>0$ there exists a function $v_\de\in H^1(\Omb)$, which is a weak solution of the equation $\Lc v_\de=0$ in the $\de$-neighborhood $\Kb_\de$ of $\Kb,$ satisfies the estimate $|\nabla v_\de|\le c_2 \frac{\om(\de)}{\de}$ in $\Kb_\de$, and it approximates $f$ on $\Kb$:
\begin{equation*}
  |f(X)-v_\de(X)|\le c_3 \om(\de), \, X\in\Kb.
\end{equation*}
\end{thm}
It stands to reason that if the set $\Kb$ is $\theta$-regular for $\theta=m-2$, the Condition \textbf{A} is satisfied, therefore, the inference of Theorem \ref{Main.Theorem} is valid.

\subsection{The  approximation}\label{sect.approximation}
Having $\de$ fixed,
we introduce the function
\begin{equation}\label{v(X)}
  v_\de(X)=\int\limits_{\Omb\setminus \Bc(\de)}\langle\aF(Y)\nabla f_0(Y), \nabla_Y G(X,Y)\rangle dY.
\end{equation}
The property of main importance for our studies is the following.
\begin{lem}\label{Lem.harmon}
  The function $v_\de$ is a weak solution of the equation $\Lc v_\de=0$ in $\Bc(\de).$
\end{lem}

\begin{proof}
  By the definition of a weak solution, see \eqref{weak} with $\mu=0$ in $\Bc(\de)$, it suffices to show that for any function $\vp\in C_0^{\infty}(\Bc(\de)),$
  \begin{equation*}
    \ab[v_\de,\vp]=0,
  \end{equation*}
  where $\ab$ is the sesquilinear form \eqref{sesq.form}.
  So, we take an arbitrary function  $\vp\in C_0^{\infty}(\Bc(\de))$ and obtain

  \begin{gather}\label{Sol1}
     \ab[v_\de,\vp]=\int_{\Bc(\de)}\langle\aF(X)\nabla_X v_\de(X),\nabla \vp(X)\rangle dX= \\\nonumber
    \int_{\Bc(\de)}\langle \aF(X)\nabla_X\left(\langle \int_{\Omb\setminus\Bc(\de)}\aF(Y)\nabla_Y f_0(Y), \nabla_Y G(X,Y)\rangle dY\right),\nabla\vp(X)\rangle dX=\\\nonumber
    \int_{\Bc(\de)}\langle \aF(X)\left(\langle \int_{\Omb\setminus\Bc(\de)}\aF(Y)\nabla_Y f_0(Y), \nabla_Y \nabla_X G(X,Y)\rangle dY\right),\nabla\vp(X)\rangle dX.
  \end{gather}

  We change the order of integration in \eqref{Sol1}: 

  \begin{equation*}
     \ab[v_\de,\vp]=\int_{\Omb\setminus\Bc(\de)}\langle\aF(Y)\nabla f_0(Y), \nabla_Y \ab[G(.,Y),\vp(.)]\rangle dY
  \end{equation*}
  Due to the main property of the Green function  (Theorem 1.1 in \cite{GW}, see \eqref{Green}), we have
  \begin{equation*}
    \ab[G(.,Y),\vp(.)]=\vp(Y).
  \end{equation*}
  Therefore,
  \begin{equation*}
    \ab[v_\de,\vp]=\int_{\Omb\setminus\Bc(\de)}\langle \aF(Y)\nabla f_0 (Y),\nabla\vp(Y)\rangle dY,
  \end{equation*}
  and the last expression equals zero since $\vp(Y)=0$ on $\Omb\setminus\Bc(\de).$
\end{proof}

\subsection{The approximation estimate}
Next, for $X_0\in \Kb$, it follows from  \eqref{v(X)} that
\begin{equation}\label{f-v}
  f(X_0)-v_\de(X_0)=f_0^*(X_0)-v_{\de}(X_0)=
  \int\limits_{\Bc(\de)}\langle\aF(Y)\nabla f_0(Y),\nabla_Y G(X_0,Y)\rangle dY.
\end{equation}
To estimate the quantity in \eqref{f-v}, we set again $r_n=2^n\de,$ $n=0,1,\dots,$ denote $\Cc_n=\Bc(\de)\bigcap(B_{r_n}(X_0)\setminus B_{r_{n-1}}(X_0)), $ and consider  the decomposition
\begin{gather}\label{int Bc}
 \int_{\Bc(\de)}\langle\aF(Y) \nabla f_0(Y),\nabla_YG(X_0,Y)\rangle dY=\\\nonumber
  \int_{\Bc(\de)\cap B_{r_0}(X_0)} \langle\aF(Y)\nabla f_0(Y),\nabla_YG(X_0,Y)\rangle dY\\\nonumber
  +\sum_{n=1}^{\infty} \int_{\Cc_n} \langle\aF(Y)\nabla f_0(Y),\nabla_YG(X_0,Y)\rangle dY.
\end{gather}
We estimate separately the terms in \eqref{int Bc}. First, according to  estimate \eqref{deriv.1G} for the derivative of the Green function,
\begin{gather}\label{int1}
\left| \int_{\Bc(\de)\cap B_{r_0}(X_0)} \langle\aF(Y)\nabla f_0(Y),\nabla_YG(X_0,Y)\rangle dY\right|\le\\\nonumber
\int_{B_{r_0}(X_0)}|Y-X_0|^{1-m}\frac{\om(d(Y))}{d(Y)}dY\le c\om(\de),
\end{gather}
by Lemma \ref{Lem3}. Further, we have
\begin{gather}\label{int2}
  \left|\int_{\Cc_n} \langle \aF(Y) \nabla f_0(Y),\nabla_Y G(X,Y)
  \rangle dY\right|\le
  c r_{n-1}^{1-m}\int_{\Cc_n}\frac{\om(\de(Y))}{\de(Y)}dY \le \\\nonumber
  cr_{n-1}^{1-m}\int\limits_{\Bc(\de)\bigcap B_{r_{n}(X_0)}}\frac{\om(d(Y))}{d(Y)}dY\le c r_n^{1-m}r_n^{m-2}\de\om(\de)=cr_n^{-1}\de\om(\de)=c2^{-n}\om(\de),
\end{gather}
the last inequality following from \eqref{L.2.estimate}. As a result, the sum on the last line in \eqref{int Bc} is bounded by
\begin{equation}\label{int3}
  c\sum_{n=0}^\infty 2^{-n}\om(\de)=c \om(\de).
\end{equation}
We, finally, collect \eqref{int Bc}, \eqref{int1}, \eqref{int2}, \eqref{int3}, which gives us
\begin{equation}\label{appr.estim}
  |f(X_0)-v_\de(X_0)|\le c \om(\de).
\end{equation}

\subsection{Gradient estimate}\label{sect.extended}
For any points $X$ in the $\de$-neighborhood $\Kb_\de$ of $\Kb$ and $Y\in \partial\Bc(\de),$ we have $|Y-X|\ge (c_1-1)\de,$  by the triangle inequality.

We set here $r_n=3c_1 2^n\de,$ $n=0, 1,\dots.$ For $X\in\Kb_\de,$ we estimate derivatives of the function $v_\de(X)$ (defined  in \eqref{v(X)}), starting with
\begin{equation}\label{grad v.1}
  \nabla_X v_{\de}(X) =\int\limits_{\Omb\setminus\Bc(\de)}\AF[X,Y]dY,
\end{equation}
where the integrand $\AF[X,Y]$ denotes
\begin{equation*}
  \AF[X,Y]\equiv \nabla_X \langle \aF(Y)\nabla f_0(Y),\nabla_Y G(X,Y)\rangle=\langle\aF(Y)\nabla_Yf_0(Y),\nabla_X\nabla_Y G(X,Y)\rangle.
\end{equation*}

Therefore,
\begin{equation}\label{grad v.2}
  |\nabla_X v_\de(X)|\le  \left|\int\limits_{(\Omb\setminus\Bc(\de))\cap B_{r_0}(X_0)} \AF[X,Y]dY\right|+
  \sum_{n=1}^\infty\left|\int\limits_{\Sc_n} \AF[X,Y]dY\right|,
\end{equation}
where $X_0\in \Kb$ is (one of) the point(s) in $\Kb$ lying on the distance $\de$ from $X$ and $\Sc_n$ is a shorthand for
\begin{equation*}
 \Sc_n\equiv (\Omb\setminus\Bc(\de))\cap (B_{r_n}(X_0)\setminus B_{r_{n-1}}(X_0))
\end{equation*}
We recall the estimate \eqref{deriv.2G} for the double  gradient of the Green function. It implies
\begin{equation*}
  |\nabla_X\nabla_YG(X,Y)|\le c r_0^{-m}, \, \mbox{for}\, Y\in(\Om\setminus\Bc(\de))\cap B_{r_0}(X_0)
\end{equation*}
and
\begin{equation*}
   |\nabla_X\nabla_YG(X,Y)|\le c  r_n^{m},\,\mbox{for}\, Y\in \Sc_n.
\end{equation*}
As a result, the inequalities \eqref{grad v.1}, \eqref{grad v.2} and the estimate \eqref{Ext2} for the gradient of the function $f_0$ give us
\begin{gather*}
  |\nabla v_\de(X)|\le c\left(r_0^{-m}\int_{B_{r_0}(X)}\frac{\om(d(Y))}{d(Y)}dY+\sum_{n=1}^{\infty}r_n^{-m}\int\limits_{B_{r_n}(X)}\frac{\om(d(Y))}{d(Y)}dY\right)\le \\\nonumber c\left(r_0^{-m}\ro_0^{m-1}\om(r_0)+\sum_{n=1}^{\infty} r_0^{-m}r_n^{m-1}\om(r_n)\right)=\\\nonumber
  c\left(\frac{\om(r_0)}{r_0}+\sum_{n=1}^{\infty} \frac{\om(r_n)}{r_n}\right)=\frac{c}{\de}\left(\om(3c_1\de)+\sum_{n=1}^{\infty}\frac{\om(3\cdot 2^n \de)}{3\cdot 2^n c_1}\right)\le c \frac{\om(\de)}{\de},
\end{gather*}
due to the  calculation as in the last line in \eqref{Sum.In}.

The results of the last two subsections show that the function $v_\de$ does in fact satisfy all requirements claimed; this concludes the proof of   our main theorem.


\begin{thebibliography}{20}
\bibitem{AlSh}T. Alexeeva, N. Shirokov, \emph{ Constructive description of Hölder-like classes on an arc in $\R^3$ by means of harmonic functions.} J. Approx. Theory, \textbf{249} (2020), 105308.
    \bibitem{Andri}V.  Andrievskii, V. Belyi,  V. Dzjadyk. Conformal Invariants in Constructive Theory of Functions of Complex Variable.  Advanced Series in Mathematical Science and Engineering, 1. World Federation Publishers Company, Atlanta, GA, 1995.
\bibitem{Browder}F. Browder,  \emph{Functional analysis and partial differential equations.} II. Math. Ann. \textbf{145} (1961/62), 81--226.
\bibitem{Browder1} F. Browder, \emph{Approximation by solutions of partial differential equations. }Amer. J. Math. \textbf{84} (1962), 134--160.
\bibitem{Bur}V. Burenkov. \emph{Mollifying operators with variable step and their application to approximation by infinitely differentiable functions.} Nonlinear Analysis, Function Spaces and Applications, Vol. 2 (Písek, 1982), pp. 5--37, Teubner-Texte zur Mathematik, \textbf{49}, Teubner, Leipzig, 1982.
\bibitem{DS1}G. David,  S. Semmes. Analysis of and on Uniformly Rectifiable Sets.
Surveys and Monograhs \textbf{38}, Amer. Math. Soc., 1993.

\bibitem{Dyn76}E.  Dyn'kin, \emph{Pseudoanalytic extensions of smooth functions. The uniform scale,} Am. Math. Soc.
 Transl. (2) \textbf{115} (1980), 33--58.
\bibitem{Dzyadyk}V. Dzyadyk, I. Shevchuk,
Theory of Uniform Approximation of Functions by Polynomials.
Translated from the Russian. Walter de Gruyter GmbH \& Co., Berlin, 2008.
\bibitem{Falconer}K. Falconer. Techniques in Fractal Geometry. Wiley. 1997.
\bibitem{Federer} H. Federer. Geometric Measure Theory. Springer, 1996.
\bibitem{Gardiner}S. Gardiner,
Harmonic Approximation.
London Mathematical Society Lecture Note Series, 221. Cambridge University Press, Cambridge, 1995.
 \bibitem{Gauthier}P. Gauthier, P. Paramonov,  \emph{Approximation by solutions of elliptic equations and extension of subharmonic functions.} New Trends in Approximation Theory, 71--87, Fields Inst. Commun., \textbf{81}, Springer, New York, 2018.
\bibitem{GW}M. Gr\"uter, K.--O. Widman, \emph{The Green function for uniformly elliptic equations.} Manuscr. Math \textbf{37}, 303--342 (1982) .
\bibitem{JW}A. Jonsson, H. Wallin: Function Spaces on Subsets of $\R^n$. Harwood Academic Publishers, 1984.
\bibitem{LSW}W. Littman, G. Stampacchia, H. Weinberger. \emph{Regular points for elliptic equations with discontinuous coefficients. } Ann. Sc. Norm. Sup. Pisa. Cl. Scien 3-e  ser., \textbf{17},  1--2, (1963), 43--77.
\bibitem{MatBook}P. Mattila. Geometry of Sets and Measures in Euclidean Spaces, Cambridge University Press, 1995.
\bibitem{Mattila}P. Mattila, P. Saaranen. \emph{Ahlfors-David regular sets and bilipschitz maps.} Ann. Acad. Sci. Fenn. Math. 34 (2009), no. 2, 487--502.
\bibitem{McShane} E. McShane. \emph{Extension of range of functions.}  Bull. Amer. Math. Soc. \textbf{40} (1934), no. 12, 837--842.
    \bibitem{Miranda}C. Miranda, Partial Differential Equations of Elliptic Type. Second revised edition. Ergebnisse der Mathematik,  \textbf{2}. Springer-Verlag, New York-Berlin, 1970.
        \bibitem{Paramonov} P. Paramonov,\emph{ Criteria for $C^1$-approximability of functions on compact sets in $\R^N,\, N\ge 3$,  by solutions of second-order homogeneous elliptic equations.}(Russian) Izv. Ross. Akad. Nauk Ser. Mat. \textbf{85} (2021), no. 3, 154--177; translation in Izv. Math. \textbf{85} (2021), no. 3, 483--505.
       \bibitem{Pavlov}D. Pavlov,  \emph{Constructive description of Hölder classes on some multidimensional compact sets.} (Russian) Vestn. St.-Peterbg. Univ. Mat. Mekh. Astron. \textbf{8(66)} (2021), no. 3, 430--441.
\bibitem{stein}E.M. Stein. Singular Integrals and Differentiability Properties of Functions. Princeton Univ. Press. 1970.
    \bibitem{Tamrazov}P. Tamrazov,  Smoothnesses and Polynomial Approximations (in Russian).  "Naukova Dumka'', Moscow, 1975.
\bibitem{Whitney}H. Whitney, \emph{Analytic extensions of functions defined on closed sets,} Trans. Amer. Math.
Soc., \textbf{36} (1934), 63--89.


\end{thebibliography}
\end{document}